\documentclass[12pt]{article}
\usepackage[utf8]{inputenc}
\usepackage{xcolor}
\usepackage[english]{babel}
\usepackage{amsmath}
\usepackage{amsthm}
\usepackage{amsfonts}
\usepackage{amsthm}
\usepackage{esint}
\usepackage{amssymb} 
\usepackage{setspace}
\usepackage[normalem]{ulem}

\usepackage{comment}

\usepackage{a4wide} 		
\usepackage{graphicx}
\usepackage{lmodern}
\usepackage[T1]{fontenc}
\usepackage[affil-it]{authblk}
\usepackage{epstopdf}

\usepackage[linkcolor=red,colorlinks=true]{hyperref}
\let\originaleqref=\eqref
\numberwithin{equation}{section}
\renewcommand{\eqref}{equation~\originaleqref}

\newcommand{\figref}[1]{Figure~\ref{#1}}

\newtheorem{theorem}{THEOREM}[section]

\newcommand{\oneplot}[1]{
    \centering\includegraphics[width=0.48\textwidth]{#1}
}
\newcommand{\twoplot}[2]{
    \centering
    \oneplot{#1}
    \hfill
    \oneplot{#2}
}

\usepackage{todonotes}
\renewcommand{\div}{\operatorname{div}\,}
\newcommand{\R}{\mathbb{R}}
\renewcommand{\Re}{\operatorname{Re}} % Reynolds number
\newcommand{\Of}{\Omega_F} % fluid domain
\newcommand{\Os}{\Omega_S} % solid domain
\newcommand{\bOf}{\bar{\Omega}_F} % fluid domain, closure of
\newcommand{\bOs}{\bar{\Omega}_S} % solid domain, closure of
\newcommand{\twoparam}[2]{{#1}_{#2}} % penalized viscosity
\newcommand{\de}{\partial}

\newcommand{\viscositynum}{m}
\newcommand{\volumenum}{n}
\newcommand{\viscosity}{\viscositynum}
\newcommand{\volume}{\volumenum}
\newcommand{\mixed}{\viscosity \lor \volume}
\newcommand{\um}{\twoparam{u}{\viscosity}} 
\renewcommand{\pm}{\twoparam{p}{\viscosity}} 
\newcommand{\vm}{\twoparam{v}{\viscositynum}} 
\newcommand{\un}{\twoparam{u}{\volume}} 
\newcommand{\pn}{\twoparam{p}{\volume}} 
\newcommand{\umn}{\twoparam{u}{\mixed}} 
\newcommand{\vmn}{\twoparam{v}{\mixed}}
\newcommand{\pmn}{\twoparam{p}{\mixed}} 
\newcommand{\app}{\text{\textsc{app}}}

\newif\ifdraft
\draftfalse % or \drafttrue

\title{
Comparative Analysis of Obstacle Approximation Strategies for the Steady Incompressible Navier-Stokes Equations}
\author{Piotr Krzyżanowski$^*$, Sadokat Malikova$^*$,  Piotr Bogusław Mucha$^*$, Tomasz Piasecki\thanks{Institute of Applied Mathematics and Mechanics, University of Warsaw, ul. Banacha 2, 02-097 Warszawa, Poland} \thanks{Corresponding author: tpiasecki@mimuw.edu.pl}}

\begin{document}

\maketitle 

% \PK{Journals to consider: \url{https://www.springer.com/journal/10440} or \url{https://www.global-sci.org/jcm/} (the latter seems to me too much oriented towards numerical methods to be suitable)}
% \PBM{What about \url{https://www.springer.com/journal/245}? it has better numbers}

%\centerline{Working questions:}
%\TP{We have inconsitency in the boundary conditions. In theoretical results we assume $u=0$ on the boundary, while in the simulations we have inflow part with nonzero velocity and do-nothing condition on the outflow. It would be better make it consistent.
% \PK{At least we would not be alone. \cite{angot2} does exactly the same thing. On the other hand, \cite{jorge} uses mixed bc from the very start.}

\medskip

\begin{abstract}
% We compare different types of obstacle approximation for the steady incompressible Navier-Stokes equations. We compare a standard penalization approximation with approximation by high viscosity in the obstacle region and composition of both methods. For all cases we provide analytical results on the rate of convergence and results of numerical experiments.    
This paper aims to compare and evaluate various obstacle approximation techniques employed in the context of the steady incompressible Navier-Stokes equations. Specifically, we investigate the effectiveness of a standard volume penalization approximation and an approximation method utilizing high viscosity inside the obstacle region, as well as their composition. Analytical results concerning the convergence rate of these approaches are provided, and extensive numerical experiments are conducted to validate their performance.
\end{abstract}


\subsubsection*{Acknowledgements} The third (PBM) and fourth (TP) author have been partly supported by the Narodowe Centrum Nauki (NCN) grant No 2022/45/B/ST1/03432 (OPUS).

\newpage 

\section{Introduction} 
Let us consider a domain $\Omega\subset \R^d$ with a physically reasonable dimension $d=2$ and $3$, in which a solid obstacle $\Os$ is immersed. Let us assume that the remaining part of the region, $\Of=\Omega\setminus \bOs$, is occupied by a viscous incompressible fluid, whose motion is governed by the Navier--Stokes equations:
\begin{equation}
\label{NS}
\begin{split}
 - \nu\Delta u  + (u\cdot \nabla)u + \nabla p  &= f \quad \text{ in } \Of,\\
\div u&=0 \quad \text{ in } \Of,\\
u &= 0 \quad \text{ on } {\partial \Of}.
\end{split}
\end{equation}
Here, $u:\Of\to \R^d$ is the velocity of the fluid and $p:\Of\to \R$ denotes the pressure. Positive constant $\nu$ is the kinematic viscosity and $f:\Of\to \R^d$ is the external force. We assume a no-slip boundary condition on the fluid--solid interface $\Sigma = \bOf \cap \bOs$ (cf.~\figref{domain})
and, for simplicity, the same homogeneous Dirichlet boundary condition on the fluid velocity~$u$ on other parts of $\Of$ as well. We will refer to \eqref{NS} as the ``real obstacle'' problem with constant viscosity.
%
%\TP{I would write here that we assume $\nu=1$ and then write $\nu=1$
%in (\ref{NSvol}) to make it consistent with (\ref{NSvisc}) and definition of $\mu_m$. In experiments we alse have $\nu=1$.}

The central inquiry under consideration in this paper revolves around the effectiveness of approximating the problem (\ref{NS}) through the utilization of a suitable system defined over the entire domain $\Omega$, rather than confining it solely to $\Of$.
\begin{figure}%[!tbp]
\centering
\includegraphics[width=0.32\textwidth]{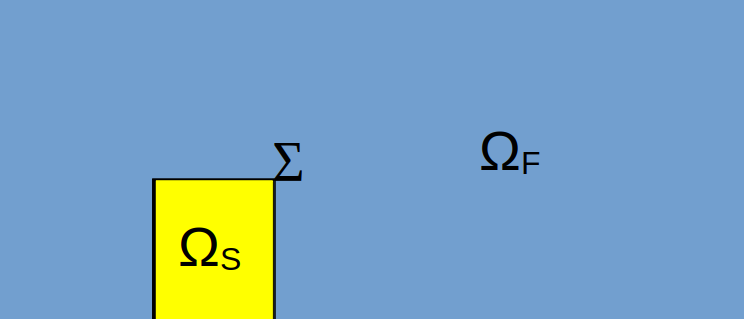}
\caption{Decomposition of $\Omega$ into fluid domain $\Of$ and solid obstacle $\Os$. The interface between the fluid and the solid is denoted by $\Sigma$ and marked with a solid black line.}
\label{domain}
\end{figure}
The proposed concept is rather straightforward: within the solid region $\Os$, a penalization term is introduced, which, for a pertinent parameter value, renders the system comparable to the original (\ref{NS}). The volume penalization method stands out as the most widely known and scrutinized technique in this regard, as it incorporates a friction term over the set $\Os$. Consequently, as the friction parameter tends towards infinity, we attain, at least formally, the equivalent of (\ref{NS}). Nonetheless, a notable concern arises regarding the convergence of such an approach, as it results in the diminution of the $L^2$-norm of the solution at the obstacle. Consequently, certain phenomena associated with the shape of the body immersed in the fluid may not be accurately captured. To address this issue, we undertake a more comprehensive examination of the viscosity penalization method. In this case, a high viscosity coefficient is imposed inside the region $\Os$, which ultimately yields the equivalent of (\ref{NS}) in the limit. From the perspective of weak solutions, this method exhibits a more natural behavior. However, it necessitates meticulous attention due to the intricate convergence analysis that is sought.

Consequently, we investigate three distinct types of such approximations (where, with a slight abuse of notation, $f:\Omega\to \R^d$ represents the force term of \eqref{NS}, extended by zero outside $\Of$).

% The main question which we want to put in considerations is: 

% {\it can the problem (\ref{NS})  be effectively approximate by a suitable system defined in the whole $\Omega$, not just in $\Of$?}

% The idea is quite simple, in the region of solid $\Os$, we want to put a penalization term, which for some relevant parameter makes the system close to the original one (\ref{NS}). The most known and examined is the volume penalization method, which introduce a friction term over the set $\Os$. Then as the friction parameter goes to the infinity, then we at least formally get (\ref{NS}). The problem is the convergence of such approach gives the 
% vanishing of the $L^2$-norm of the solution at the obstacle, so in particular some behaviors related to the shape of the body immersed in the fluid can be not captured. From that reason we want to analyze deeper the case of the viscosity penalization method. In that case
% we put at the region $\Os$ high viscosity coefficient, which at limit also gives us (\ref{NS}). At level of weak solutions this method seems to be more natural, however required more care as we want to consider more advance convergence analysis.

%In the paper we develop on the idea to extend the system of PDEs onto whole $\Omega$, with certain penalization terms acting as a means to take into account --- though only approximately --- the presence of the obstacle. 
% Thus, we consider three types of such approximations (where, with a little abuse of the notation, $f:\Omega\to \R^d$ denotes the force term of \eqref{NS} extended by zero outside $\Of$):

\paragraph{Volume penalization}  One popular approach augments the momentum equation with a penalization term, aiming at slowing down the fluid inside the obstacle region. With this approach, one aims at suppressing the motion of the fluid in the obstacle region by introducing inside $\Os$ damping, controlled by (large enough) parameter $n\geq 0$, so that the approximate solution $(\un,\pn)$ satisfies
\begin{equation} 
\label{NSvol}
\begin{split}
 - \nu\Delta \un  + (\un\cdot \nabla)\un + \nabla \pn + \eta_n \un &= f \quad \text{ in } \Omega,\\
\div \un&=0 \quad \text{ in } \Omega,\\
\un &= 0 \quad \text{ on } {\partial \Omega},
\end{split}
\end{equation}
where $\eta_n$ is a nonnegative piecewise constant function depending on a penalty parameter $n\geq 0$, 
$$
\eta_n(x) = \begin{cases}
0, &\quad x\in\Of,\\
n, &\quad x\in\Os.\\
\end{cases}
$$
% Obviously, the solution $(\un,\pn)$ implicitly depends on the penalty parameter $n$, but to keep the notation clean we will not indicate this directly while referring to it.

Since the method essentially treats the obstacle as a porous medium whose permeability coefficient is proportional to $1/n$, it is sometimes called Brinkman penalization; in other papers it is referred to as the $L^2$ penalization. Here we mention a number of works based on a volume penalization method such as \cite{jorge}, \cite{angot2}, \cite{angot},\cite{kolomenski}. In works of Angot \cite{angot2}, and Angot, Bruneau and Fabrie  \cite{angot}, authors established the strong convergence of the solutions and derived some error estimates for the approximate and exact problem in the steady Stokes system and unsteady Navier-Stokes with homogeneous boundary data. The further analysis of error estimates for steady systems with inhomogeneous boundary conditions were done recently by Aguayo and Lincopi \cite{jorge}.

\paragraph{Viscosity penalization} Similar effect may be realized by introducing a very large artificial viscosity in $\Os$ instead --- though this approach is viable only if the obstacle is pinned to the boundary of $\Omega$. The approximate solution $(\um,\pm)$ is defined on $\Omega$ by the Navier--Stokes system
\begin{equation}
\label{NSvisc}
\begin{split}
 - \div(\mu_m\nabla \um)  + (\um\cdot \nabla)\um + \nabla p_m  &= f \quad \text{ in } \Omega,\\
\div \um&=0 \quad \text{ in } \Omega,\\
\um &= 0 \quad \text{ on } {\partial \Omega},
\end{split}
\end{equation}
where $\mu_m$ is a positive, piecewise constant function depending on a penalty parameter $m>0$,
$$
\mu_m(x) = \begin{cases}
\nu, &\quad x\in\Of,\\
m \, \nu, &\quad x\in\Os.\\
\end{cases}
$$
% Again, the value of $m$ which the solution $(\um,\pm)$ depends on will always follow indirectly from the context. 

The viscosity penalization method was developed by Hoffmann and Starovoitov \cite{hoffman}, and San Martin \textit{et al} \cite{sanmartin}. This method was used in works of Wr\'{o}blewska-Kami\'{n}ska \cite{aneta},  and Starovoitov \cite{starovoitov} to construct solutions to systems describing motion of rigid bodies immersed in incompressible fluid. 
%which is based on the idea of approximating rigid objects of the system by the fluid of very high viscosity becoming singular in limiting consideration. 
Recently, in \cite{malikova},   a higher regularity result was proved for weak solutions to (\ref{NSvisc}). 

\paragraph{Mixed penalization} The third possibility that we will consider here is simply a combination of the two above mentioned  approaches, leading to a system of the form 
\begin{equation} 
\label{NSmix}
\begin{split}
 - \div(\mu_m\nabla \umn)  + (\umn\cdot \nabla)\umn + \nabla \pmn + \eta_n \umn &= f \quad \text{ in } \Omega,\\
\div \umn&=0 \quad \text{ in } \Omega,\\
\umn &= 0 \quad \text{ on } {\partial \Omega},
\end{split}
\end{equation}
where $\nu_m,\eta_n$ are defined as above and $(\umn,\pmn)$ denote the approximate solution. Clearly, this formulation covers the previous ones as special cases: if $n=0$, then $\umn\equiv\um$
% of \eqref{NSvisc} 
and similarly, for $m=1$ there holds $\umn\equiv\un$. % from \eqref{NSvol}.
Penalization of mixed type has been considered among others in \cite{angot2}, \cite{angot} for $m=n$;  the second includes also penalization of the time derivative. The authors provide theoretical results on the rate of convergence and their numerical validation. 

% Our goal is to continue in this direction providing analytical results on the rate of convergence for both above mentioned methods of approximation, as well as a mixed approach consisting in combination of both methods.
% The analysis is done for the two and three dimensional cases.
% We verify the rate of convergence numerically in the two dimensional case. Our analysis is performed for different values of Reynolds number and different shapes of the obstacle.   

\smallskip

The primary objective of our study is to advance in this research direction by offering analytical insights into the convergence rate of the aforementioned obstacle approximation methods, including their combined approach. Our analysis encompasses both two-dimensional and three-dimensional scenarios.

To validate the theoretical findings, we conduct a comprehensive numerical investigation to assess the convergence rate in the two-dimensional case. The numerical experiments are designed to encompass
% a range of Reynolds numbers and 
all above mentioned penalization approaches and diverse obstacle shapes as well, ensuring a robust evaluation of the proposed approximation methods.

\begin{figure}%[!tbp]
\twoplot{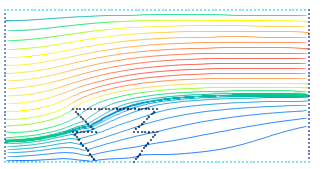}{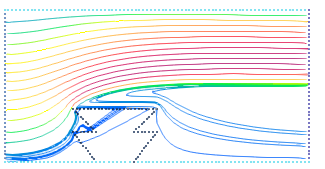}\\
\twoplot{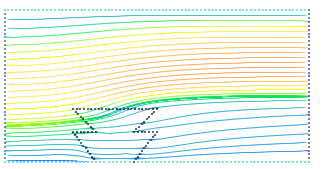}{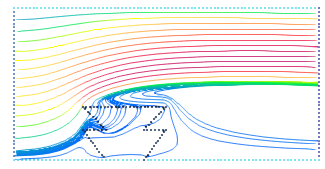}
\caption{Streamlines plots, with color corresponding to velocity magnitude, for viscosity penalization (top row)  and for volume penalization (bottom row) and varying penalty parameter. Panels on the left correspond to penalty parameter $m=n=10^3$; those on the right, to $m=n=10^5$. (For details regarding the experimental setting, see Section~\ref{sec:experim:taipei})}
\label{showcase}
\end{figure}

By undertaking this analytical and numerical analysis, we aim to contribute to a deeper understanding of the performance and efficacy of these approximation techniques in capturing the behavior of fluid flow around obstacles. This research has the potential to inform and guide practitioners in selecting the most suitable approach based on the Reynolds number and the specific geometric characteristics of the obstacles encountered in practical applications.

% As a practical  application one can finally we demonstrate at the following example. For difference methods we obtain
% a few possible scenarios. Above we see that the methods based on the viscous penalization methods seems to me the most appriopate. 

As a practical application, we present a demonstrative example (\figref{showcase}) to showcase the performance of two  approximation methods considered here. %By employing various penalization methods, we explore several potential scenarios.  
From this experiment it follows that while viscosity penalization with $m=10^5$ already results in a quite plausible flow, for the same penalty value $n=10^5$ one still gets visible non-physical artifacts when using volume penalization. From our subsequent analysis, it becomes evident that the approach based on the viscous penalization method exhibit favorable error characteristics. Further experiments, presented in Section \ref{sec:num}, 
show that a combination of both methods --- i.e., (\ref{NSmix}) --- leads to further improvement of the accuracy.  
%\PK{The part of this sentence which is below sounds an overly bold statement. (Note also that in terms of the error the best performer is the mixed penalization.)}
%and appear to be the most suitable for the given application. 

This observation highlights the effectiveness and suitability of obstacle approximation techniques based on the incorporation of high viscosity within the obstacle region. Numerical results provide evidence supporting the practical viability of these methods in accurately capturing fluid flow behavior around obstacles. Clearly, to develop a robust numerical solver based on the penalization idea, one would need to put a significant amount of further effort. In particular, the ill-conditioning resulting from very high contrast in the coefficients would require one to employ an efficient preconditioner; otherwise the overall performance of the solver will suffer. Several promising preconditioners for high contrast Stokes equations have recently been developed (see \cite{Wichrowski2022AMM} for an example) and may possibly be adapted to fit the present framework. However, in this paper we do not investigate such practical issues, and use numerical simulations only as a means to verify the quality of the approximation and sharpness of theoretical estimates.

The findings of this paper contribute to the understanding of 
%\del{the obstacle}
penalty-based approximation methods, offering valuable insights for researchers and practitioners seeking to tackle real-world fluid flow problems involving obstacles. The outcomes of this study can assist in making informed decisions regarding the selection of appropriate approximation strategies for specific applications.

\subsection{Weak formulation}

We shall assume that $\Omega$, $\Of$ and $\Os$ are bounded, open domains in $\R^d$, and that their boundaries are sufficiently regular. As mentioned above, $\bar{\Omega} = \bOf \cup \bOs$, while  $\Of \cap \Os = \emptyset$. We also assume that the interface $\Sigma = \bOf \cap \bOs$ has a positive measure.

As we are working with weak solutions, we shall recall the weak formulation of the original and approximate problems. For this 
purpose we define
\begin{equation} \label{def:V} 
V_F = \{v \in H^1_0(\Of): \div v=0 \} \qquad\text{and}\qquad  
V = \{v \in H^1_0(\Omega): \div v=0 \}.
\end{equation}
Assume $f \in V_F'$. A weak solution to (\ref{NS}) is $u \in V_F$ such that
\begin{equation}\label{weak}
  \int_{\Of}  (u\cdot \nabla) u\, v +\\
\nu \nabla u:\nabla v \,dx = \langle f,v \rangle_{V_F',V_F} 
%\int_{\Of} f\cdot v\,dx  
\end{equation} 
holds for any $v \in V_F$, where $\langle\cdot,\cdot \rangle$ is the duality pairing.
%We shall call it the reference solution.

%\PK{If the above is called "the reference solution", maybe it would be reasonable to replace (in several other places) the notion of "the real obstacle flow" with "the reference solution" as well? (I am not insisting on it, just asking.)}

In order to define weak solutions to approximate problems we assume $f \in V'$. Then a weak solution to the viscosity penalization  \eqref{NSvisc} is $\um \in V$ such that
\begin{equation}\label{weak_visc}
  \int_{\Omega}  (\um\cdot \nabla )\um\, v+\\
\mu_m{\nabla \um}:\nabla v\,dx = \langle f,v \rangle_{V',V} 
\end{equation}
holds for any $v \in V$. 

Next, by a weak solution to the volume penalization problem (\ref{NSvol}) we mean 
$u_n \in V$ s.t.
\begin{equation}\label{weak_fric}
  \int_{\Omega}  (\un\cdot \nabla)\un\, v + \\
\nu {\nabla \un}:\nabla v +  \eta_n \un v \,dx = \langle f,v \rangle_{V',V} 
\end{equation}
holds for every $v \in V$.

Finally, a weak solution to the mixed penalization system \eqref{NSmix} is $\umn \in V$ such that 
\begin{equation}\label{weak_mix}
   \int_{\Omega}  (\umn\cdot \nabla)\umn v +\\
\mu_m{\nabla \umn}:\nabla v +  \eta_n \umn v \,dx = \langle f,v \rangle_{V',V} 
\end{equation}
holds for every $v \in V$.

\section{Theoretical bounds on the convergence rate}

\renewcommand{\um}{\twoparam{u}{\viscositynum}} 
\renewcommand{\un}{\twoparam{u}{\volumenum}} 

In this section we prove convergence and upper bounds on the approximation error with respect to penalty parameters for all three approximation schemes introduced above.
% Hence, we will stress dependence of the solution on either of these parameters by denoting the approximate solution of (\ref{NSvol}) by $\un$ --- instead of  $\twoparam{u}{\volume}$ --- recalling that $n$ is the volume penalty parameter. Accordingly, we will use the symbol  $\um$ to denote $\twoparam{u}{\viscosity}$ which solves (\ref{NSvisc}) for the viscosity penalty~$m$.

\subsection{Volume penalization}

First we recall the error estimates for the approximation of the flow by means of volume penalization. For inflow condition on the velocity and $f=0$ they have been proved recently in \cite{jorge} (Theorems 5 and 6). In order to understand the statement of results we shall recall that for stationary version of the Navier-Stokes system (\ref{NS}) we are not able to require the uniqueness of solutions. This feature holds only for some restrictive cases like smallness of the external force. For that reason in the large data case, our approximation defines the original solution on a certain subsequence. 
The proof in our setting requires only minor modification, therefore we skip it.
%\TP{In \cite{jorge} it is assumed that $\de \Omega$ is piecewise $C^1$ for the weaker convergence and piecewise $C^2$ for the stronger, but this 'piecewise' is strange since then the boundary can have singularities. I think the assumptions should be as below}

\begin{theorem} \label{thm:vol}
Assume $f \in H^{-1}(\Omega)$ and $\Omega$ is a Lipschitz domain. Let  $\un$ denote a  weak solution to \eqref{NSvol}. Then 
\begin{equation} \label{vol:1.1}
\|\un\|_{L^2(\Os)} \leq Cn^{-1/2}. 
\end{equation}
Moreover, for a subsequence $u_{n_k}$, denoted in what follows $\un$,  
\begin{align}
\lim_{n \to \infty} \|u-\un\|_{H^1(\Omega)}=0 \label{vol:1.2},
%\|\un\|_{H^{1/2}(\Os)} \leq Cn^{-1/4} \label{vol:1}\\[3pt]
%\|u-\un\|_{H^{1/2}(\Of)} \leq Cn^{-1/4} \label{vol:2}
\end{align}
where $u$ is a weak solution to \eqref{NS}.
Assuming additionally that $\de \Omega \in C^{2}$,\\ $f \in L^2(\Omega)$ and $\|f\|_{H^{-1}(\Omega)}$  is small enough with respect to $\nu$ we have 
\begin{align}
&\|\un\|_{L^2(\Os)} \leq Cn^{-3/4}, \label{vol:3} \\[3pt]
&\|u-\un\|_{H^{1}(\Of)} \leq Cn^{-1/4}. \label{vol:4}    
\end{align}
\end{theorem}

\subsection{Viscosity penalization}

It turns out that the approximation by means of viscosity penalization gives rise to faster convergence, however we have to assume that the obstacle touches the boundary of $\Omega$ to ensure the convergence of the approximate solution on the obstacle domain to zero (otherwise we would only obtain a constant flow). A result of this kind is expected, but to our knowledge has not been proved so far in the stationary case. It is given in the main result of this paper, which reads 
\begin{theorem} \label{thm:visc}
Assume $\Omega$ and $\Of$ are Lipschitz domains and $f \in H^{-1}(\Omega)$. 
Assume moreover that and
\begin{equation} \label{touching}
int\,\overline{\de \Os \cap \de \Omega} = \de \Os \cap \de \Omega \quad {\rm and} \quad  \lambda_S(\de \Os \cap \de \Omega)>0,
\end{equation}
where $\lambda_S$ is the surface Lebesgue measure. 
Let $u$ be a weak solution to \eqref{NS} and $\um$ a weak solution to (\ref{NSvisc}). Then
\begin{equation}
\|\um\|_{H^1(\Os)} \leq C m^{-1/2}\nu^{-1} \label{visc:1}, 
\end{equation}
Moreover, there exists a subsequence $u_{m_k}$, which we will denote again by $u_m$, such that 
\begin{equation}
\lim_{m\rightarrow \infty}\|u-\um\|_{H^1(\Of)} = 0 \label{visc:2},
\end{equation}
where $u$ is a weak solution to \eqref{NS}. 
If additionally $\Omega$ and $\Of$ are $C^2$ domains, while $f \in L^2(\Omega)$ and $\|f\|_{H^{-1}(\Omega)}$ is sufficiently small with respect to $\nu$, then
\begin{align}
&\|\um\|_{H^1(\Os)} \leq C(\nu m)^{-1}, \label{visc:3} \\
&\|u-\um\|_{H^{1}(\Of)} \leq C \nu^{-1} m^{-1/2}. \label{visc:4}
\end{align}
\end{theorem}

\begin{proof}
Let $(u,p)\in V(\Omega)\times L^2(\Omega) $ be the solution of stationary Navier--Stokes equations
\begin{equation}
\label{NS2}
\begin{split}
 - \nu \, \Delta u_F  + (u_F\cdot \nabla)u_F + \nabla p_f &= f_F  \quad \text{ in } \Of,\\
\div u_F &= 0  \quad \text{ in } \Of, \\
u_F &= 0 \quad \text{ on } {\partial \Of}.
\end{split}
\end{equation}
with the prolongation $(u_S,p_S)=(0,0)$ in $\Os$ (see \cite{angot2}). 
Taking $\um$ as a test function in \eqref{weak_visc} we get 
\begin{equation}
   \int_{\Of}  \nu\vert{\nabla \um}\vert^2\,dx + \nu m  \int_{\Os}  \vert{\nabla \um}\vert^2\,dx = \langle f,u_m \rangle_{H^{-1}(\Omega),H^1_0(\Omega)} \leq \Vert f\Vert_{H^{-1}(\Omega)} \Vert \nabla \um\Vert_{L^2(\Omega)} 
.\end{equation}
Using Young inequality, we have
\begin{equation} \label{ene:1}
   \frac{\nu}{2}  \Vert{\nabla \um}\Vert^2_{L^2(\Of)} + \frac{\nu m}{2}   \Vert{\nabla \um}\Vert^2_{L^2(\Os)} \leq \frac{1}{2\nu}\Vert f\Vert^2_{H^{-1}(\Omega)}  
,\end{equation}
%combining with
%\begin{equation}
%    \Vert{\nabla \um}\Vert^2_{L^2(\Omega)} \leq   \int_{\Of}  \vert{\nabla \um}\vert^2\,dx + m  \int_{\Os}  \vert{\nabla \um}\vert^2\,dx   
%.\end{equation}
therefore we get
\begin{equation}\label{2.13}
     \Vert{ \nabla \um}\Vert_{L^2(\Os)}   \leq \frac{C}{m^{1/2}\nu}\Vert f\Vert_{H^{-1}(\Omega)},  
\end{equation}
which proves (\ref{visc:1}).  
%The weak form of the penalized viscosity approximate problem 
%\begin{equation}\label{weak_v_st}
%   \int_{\Of}  \nabla \um\cdot \nabla \phi\,dx + m  \int_{\Os}  \nabla \um \cdot \nabla \phi \,dx   = \langle f, \phi \rangle_{V^{\prime}, V} 
%\end{equation}
%for $\phi \in V$. 
The estimate (\ref{ene:1}) implies that there exists a subsequence, which we denote again by $u_m$, s.t. 
\begin{equation} \label{conv:1}
\um\rightarrow \Tilde{u} \in L^p(\Omega), \quad 
\um\rightharpoonup \Tilde{ u} \in H^1(\Omega) 
\end{equation}
for $1\leq p <\infty$ in case $d=2$ and $1\leq p <6$ in case $d=3$. 
Moreover, (\ref{2.13}) and (\ref{touching}) gives  $\Tilde{u} = 0$ in $\Os$.
Next, we can rewrite (\ref{weak_visc}) as 
\begin{equation*}
\nu m\int_{\Os} \nabla \um:\nabla v \,dx = \langle f,v \rangle_{V',V}
-\int_{\Of}\nu\nabla u_m : \nabla v \,dx - \int_{\Omega}(u_m\cdot\nabla)u_m v \,dx.
\end{equation*}
This identity together with (\ref{ene:1}) implies that
$- \div\left[m \chi_{\Os}\nabla \um\right]$ is bounded in $V^{\prime}$, therefore there exists $h \in V'$ such that 
$$
\lim_{m\rightarrow \infty}  (- \div m\chi_{\Os} \nabla \um) = h \;\; \textrm{weakly in} \; V^{\prime}
$$
and 
 \begin{equation} \label{h:prop}
 \langle h , \phi\rangle =0 \quad \forall \phi \in \mathcal{D}(\Of).
 \end{equation}
The convergences (\ref{conv:1}) allow to pass with $m\to\infty$ in \eqref{weak_visc} to obtain 
\begin{equation}\label{h}
\int_{\Omega} (\tilde u\cdot\nabla)\tilde uv\,dx+\nu\int_{\Omega} \nabla \Tilde{ u} \cdot \nabla v \,dx + \langle h, v\rangle_{V^{\prime}, V} = \langle f, v\rangle_{V^{\prime}, V} \quad \forall v \in V.
\end{equation}
From (\ref{h:prop}) and (\ref{h}) we conclude
\begin{equation}
\int_{\Of} (\tilde u\cdot\nabla)\tilde uv\,dx+\nu \int_{\Of} \nabla \Tilde{ u} \cdot \nabla v \,dx 
%+ \langle h, v\rangle_{V_F^{\prime}, V_F} 
= \langle f, v\rangle_{V_F^{\prime}, V_F} \quad \forall v \in V_F. 
\end{equation}
By continuity of trace operator, $\um =0$ on $\partial \Omega$ implies $\Tilde{u}|_{\partial \Omega}=0$. Therefore $\tilde u$ is indeed a weak solution to \eqref{NS}, so in what follows we write $u=\tilde u$.
%together with trace lemma we have strong convergence in $H^1(\Os )$. $\Tilde{u}$ satisfies Stokes problem in $\Of$ that is unique, thus $\Tilde{u} = u$. So, $\um$ converges weakly to $u$  in $V$.
It remains to prove the strong convergence in $H^1(\Omega)$. For this purpose we subtract (\ref{weak_visc}) and (\ref{h}) taking $v = \um- u$. We get
\begin{equation}\label{er}
\begin{split}
   \nu\int_{\Of}  \vert{\nabla (\um - u)}\vert^2\,dx + \nu m \int_{\Os}  \vert{\nabla (\um - u)}\vert^2\,dx   &= \langle h ,(\um - u) \rangle_{V^{\prime}, V}\\
   &-\int_{\Omega} [(\um\cdot\nabla)\um-(u \cdot\nabla)u](\um - u)\,dx 
.\end{split}
\end{equation}
By the weak convergence of $\um$ in $H^1(\Omega)$, the first term on the RHS of the above expression tends to zero. The second can be decomposed as 
$$
\int_{\Omega} (u_m\cdot\nabla)(u_m-u)(u_m-u)\, dx
+\int_{\Omega} ((u_m-u)\cdot \nabla )u (u_m-u)\,dx.
$$
We have $u_m(u_m-u)\rightarrow 0$ in $L_p$ for $p<3$,
which together with weak convergence of $\nabla u_m$ implies convergence of the first integral to zero.  The second converges obviously by \eqref{conv:1} and H\"older inequality.
Therefore we have (\ref{visc:2}).

Now, assume that $f\in L^2(\Omega)$ and $\Omega,\Of \in C^2$. 
The idea is to get rid of integrals over $\Of$ on the RHS of the identity leading to the energy estimate, and keep there only $\Os$ terms, which allow to take advantage of the large parameter $m$ to show higher order of convergence. 
However, in case on nonlinear system we have also a term with difference of convective terms, which enforces the assumption of smallness of $\|f\|_{H^{-1}}$.
Under the assumed regularity of $f,\Omega$ and $\Omega_F$ we have $(u_f,p_f)\in H^2(\Of)\times H^1(\Of)$,  
and therefore 
\begin{equation}
\tau_f n := -p n + \nu \nabla u\cdot n \in H^{1/2}(\Sigma).
\end{equation}
Testing the Navier-Stokes equation (\ref{NS}) with 
$\phi \in H^1_0(\Omega)$ not vanishing on $\Sigma$ 
and assuming $u = 0$ in $\Os$ we obtain
\begin{equation}\label{weak_st}
  \int_{\Of} ( u\cdot\nabla) u\phi\,dx +\nu\int_{\Of} \nabla u \cdot  \nabla  \phi   \,dx - \int_{\Sigma} g \cdot \phi ds = \int_{\Of} f \cdot\phi dx, 
\end{equation}
where $g:= \tau_f n \in H^{1/2}(\Sigma)$.  
The identity (\ref{weak_st}) holds in particular $\forall \phi \in V$, which together with (\ref{h}) gives 
\begin{equation}\label{h_hint}
    \langle h, \phi\rangle_{V^{\prime}, V} = \int_{\Os} f \cdot\phi \, dx - \int_{\Sigma} g \cdot \phi \,ds \quad \forall\;\phi \in V.
\end{equation}
Using (\ref{h_hint}) in (\ref{er}) and taking $\vm:= \um - u$ we obtain
\begin{multline}\label{error}
   \nu\int_{\Of}  \vert{\nabla \vm }\vert^2\,dx +\nu m \int_{\Os}  \vert{\nabla \vm }\vert^2\,dx   = \int_{\Os} f \cdot \vm \, dx - \int_{\Sigma} g \cdot \vm \,ds\\
   -\int_{\Omega} \bigl((\um\cdot\nabla)\um-(u \cdot\nabla)u\bigr)\vm \,dx .
\end{multline}
Now, we will estimate RHS of (\ref{error}). For this purpose we use H\"older, Young and Poincar\'e inequalities and get
\begin{equation}
    \int_{\Os} f \cdot \vm \, dx \leq \dfrac{C_P}{2\nu m} \Vert f\Vert^2_{L^2(\Os)}+\dfrac{\nu m}{2} \Vert \nabla \vm\Vert^2_{L^2(\Os)},
\end{equation}
where $C_P$ is the constant from the Poincar\'e inequality. Next,
we use H\"older inequality, trace lemma and Young inequality to obtain 
\begin{equation}
    \int_{\Sigma} g \cdot \vm \,ds \leq C \Vert g\Vert_{L^2(\Sigma)}\Vert v_{m}\Vert_{L^2(\Sigma)}\leq \dfrac{C}{\nu m} \Vert g\Vert^2_{L^2(\Sigma)} + \dfrac{\nu m}{4} \Vert \nabla \vm\Vert^2_{L^2(\Os)}
.\end{equation}
For the last term on the RHS of (\ref{error}) we have, 
again by H\"older and Poincar\'e inequalities,
\begin{multline}\label{nonlin}
\left| \int_{\Omega} [(\um\cdot\nabla)\um-(u \cdot\nabla)u]\vm \,dx \right|
\leq \int_{\Omega} \left(|(u_m \cdot \nabla v_m)v_m| + |(v_m \cdot \nabla u)v_m| \right) \,dx \\
\leq C (\|\nabla u_m\|_{L^2(\Omega)}+\|\nabla u\|_{L^2(\Omega)}) \|\nabla v_m\|^2_{L^2(\Omega)}.
\end{multline}
Combining the above inequalities, we get 
\begin{multline*}
    \nu \Vert \nabla \vm\Vert^2_{L^2(\Of)}+\dfrac{\nu m}{4}\Vert \nabla \vm \Vert^2_{L^2(\Os)} \\
    \leq  \dfrac{1}{\nu m} \Big( C(\Omega, f)+  C\Vert g\Vert^2_{L^2(\Sigma)} \Big)
    +C (\|\nabla u_m\|_{L^2(\Omega)}+\|\nabla u\|_{L^2(\Omega)}) \|\nabla v_m\|^2_{L^2(\Omega)}.
\end{multline*}
Assuming $\|f\|_{H^{-1}(\Omega)}$ sufficiently small w.r.t. $\nu$ we can absorb the last term on the RHS by the LHS to obtain 
\begin{equation}
    \nu \Vert \nabla \vm\Vert^2_{L^2(\Of)}+\dfrac{\nu m}{4}\Vert \nabla \vm \Vert^2_{L^2(\Os)} \leq \dfrac{1}{\nu m} \Big( C(\Omega, f)+  C\Vert g\Vert^2_{L^2(\Sigma)} \Big).
\end{equation}
Recalling that $\vm= \um - u$ 
and $u \equiv 0$ on $\Os$ we get
\[\Vert \um- u\Vert_{H^1(\Omega)}\leq \nu^{-1}m^{-1/2} C_1 (\Omega,g,f)\]
and 
\[\Vert \um\Vert_{H^1(\Os)}\leq (\nu m)^{-1}C_2 (\Omega,g,f).\]
\end{proof}

\subsection{Mixed penalization}
\label{sec:mixed}
It is obvious that the mixed approximation (\ref{NSmix}) satisfies the same estimates as both volume and viscosity approximations. However, it is possible to show additional estimates which involve both approximation parameters. In our proof, we will make use of the Poincar\'e inequality, therefore we will assume again that the obstacle touches the boundary (we recall however that this assumption is not necessary for the convergence of the mixed approximation --- we only need it to obtain the novel estimates).
\begin{theorem} \label{thm:mix}
Assume $\Omega$ and $\Of$ are Lipschitz domains, $f \in H^{-1}(\Omega)$ and condition (\ref{touching}) holds. Let  $\umn$ a weak solution to (\ref{NSmix}). Then 
$\umn$ satisfies (\ref{vol:1.1}) and (\ref{visc:1}). 
Moreover, there exists a subsequence, still denoted $\umn$, which satisfies the estimates (\ref{vol:1.2}) and (\ref{visc:2}), where $u$ is a weak solution to \eqref{NS}.
%\begin{align}
%\|\umn\|_{H^1(\Os)} \leq Cm^{-1/2} \label{mix:1}  %\lim_{m\rightarrow \infty}\|u-\umn\|_{H^1(\Omega)} = 0. \label{mix:2}
%\end{align}
If we assume additionally that $\Omega$ and $\Of$ are $C^2$ domains, $f \in L^2(\Omega)$ and $\|f\|_{H^{-1}(\Omega)}$ is sufficiently small with respect to $\nu$ then  
the estimates (\ref{vol:3})-(\ref{vol:4})  and (\ref{visc:3})-(\ref{visc:4}) hold. Moreover, we have 
\begin{align}
&\|\umn\|_{H^1(\Os)} \leq C(\nu m)^{-3/4}n^{-1/4}, \label{mix:3} \\[3pt]
&\|\umn\|_{L^2(\Os)} \leq C(\nu m)^{-1/4}n^{-3/4}, \label{mix:4} \\[3pt]
&\|u-\umn\|_{H^{1}(\Of)} \leq C(\nu mn)^{-1/4}. \label{mix:5}
\end{align}
\end{theorem}
\begin{proof}

%We take $n=p(m)$, where $p(m)$ is linear function.
Taking $\umn$ as a test function in \eqref{weak_mix} we get
\begin{equation} 
   \int_{\Of}  \nu \vert{\nabla \umn}\vert^2\,dx + \nu m  \int_{\Os}  \vert{\nabla \umn}\vert^2\,dx + n \int_{\Os}  \vert{ \umn}\vert^2\,dx  \leq \Vert f\Vert_{H^{-1}} \Vert \nabla \umn\Vert_{L^2(\Omega)}, 
\end{equation}
which gives the same estimates as in pure volume and viscosity approximations. The proof of strong convergence in $H^1(\Omega)$ is analogous to the viscosity case. 
%
%Using Young and Poincar\'e inequalities, we get
%\begin{equation}\label{2.30}
%     \Vert{ \nabla \umn}\Vert_{L^2(\Os)}   \leq \frac{1}{ m^{1/2}} C_1 (\Omega)\Vert f\Vert_{H^{-1}(\Omega)},  
%\end{equation}
%as $m+p(m)$ is an order of $m$, which proves (\ref{mix:1}).  
We have, up to a subsequence,  
\begin{equation} \label{conv:2}
\umn\rightarrow \Tilde{u} \in L^p(\Omega), \quad 
\umn\rightharpoonup \Tilde{ u} \in H^1(\Omega)
\end{equation}
for $1\leq p <\infty$ in case $d=2$ and $1\leq p <6$ in case $d=3$, 
where $\Tilde{u} = 0$ in $\Os$. Moreover, there exists $h \in V'$ such that 
$$
\lim_{m\rightarrow \infty}  (- \div m\chi_{\Os} \nabla \umn + n\chi_{\Os} \umn) = h \;\; \textrm{weakly in} \; V^{\prime}
$$
and 
 \begin{equation} \label{h1:prop}
 \langle h , \phi\rangle =0 \quad \forall \phi \in \mathcal{D}(\Of).
 \end{equation}
The convergences (\ref{conv:2}) allow to pass with $m\to\infty$ in \eqref{weak_mix} to obtain 
\begin{equation}\label{h1}
\int_{\Omega} (\tilde u\cdot\nabla)\tilde uv\,dx+\int_{\Omega} \nabla \Tilde{ u} \cdot \nabla v \,dx + \langle h, v\rangle_{V^{\prime}, V} = \langle f, v\rangle_{V^{\prime}, V} \quad \forall v \in V.
\end{equation}
From (\ref{h1:prop}) and (\ref{h1}) we conclude that
$\tilde u$ is indeed a weak solution to \eqref{NS}, so in what follows we write $u=\tilde u$ with $\Tilde{u}|_{\partial \Omega}=0$.

Next, we subtract (\ref{weak_mix}) and (\ref{h1}) taking $v = \umn- u$ to obtain 
%prove the strong convergence in $H^1(\Omega)$, 
\begin{multline}\label{ermix}
   \nu \int_{\Of}  \vert{\nabla (\umn - u)}\vert^2\,dx + \nu m \int_{\Os}  \vert{\nabla (\umn - u)}\vert^2\,dx + n \int_{\Os}  \vert{ (\umn - u)}\vert^2\,dx   \\
   = \langle h ,(\umn - u) \rangle_{V^{\prime}, V}-\int_{\Omega} [(\umn\cdot\nabla)\umn-(u \cdot\nabla)u](\umn - u)\,dx .
\end{multline} 
Exactly as in the proof of \eqref{visc:2} we verify that the 
%By continuity of $h$ and weak convergence of $\umn$ in $H^1(\Omega)$ 
RHS of the above expression tends to zero, which proves the strong convergence in $H^1(\Omega)$.\\
Now, assume that $f\in L^2(\Omega)$ and $\Omega,\Of \in C^2$.
%Recall \ref{weak_st} the weak form of Navier-Stokes equation, where $g:= \tau_f n \in H^{1/2}(\Sigma)$.
 Using (\ref{h_hint}) in (\ref{ermix}) and taking $\vmn:= \umn- u$ we get
\begin{equation}\label{errormix}
\begin{split}
   \int_{\Of}  &\vert{\nabla \vmn }\vert^2\,dx + m \int_{\Os}  \vert{\nabla \vmn }\vert^2\,dx + n \int_{\Os}  \vert{ \vmn }\vert^2\,dx  \\
   &= \int_{\Os} f \cdot \vmn \, dx - \int_{\Sigma} g \cdot \vmn \,ds-\int_{\Omega} [(\umn\cdot\nabla)\umn-(u \cdot\nabla)u]\vmn \,dx 
.\end{split}
\end{equation}
Let us estimate the RHS of (\ref{errormix}). By  H\"older, Young and Poincar\'e inequalities we get
\begin{equation}
\begin{aligned}
    \int_{\Os} f \cdot \vmn \, dx & \leq \frac{C}{(\nu mn)^{1/2}}\|f\|^2_{L^2(\Os)}+\frac{1}{2}(\nu mn)^{1/2}\|\nabla \vmn\|_{L^2(\Os)}\|\vmn\|_{L^2(\Os)}\\[3pt]
    & \leq \frac{C}{(\nu mn)^{1/2}} \Vert f\Vert^2_{L^2(\Os)}+\frac{\nu m}{4} \Vert \nabla \vmn\Vert^2_{L^2(\Os)}
    +\frac{n}{4} \Vert  \vmn\Vert^2_{L^2(\Os)}.
\end{aligned}
\end{equation}
Next, we use H\"older inequality, trace theorem, Young and interpolation inequalities to obtain 
\begin{align*}
    \int_{\Sigma} g \cdot \vmn \,ds & \leq \|g\|_{L^2(\Sigma)}\|\vmn\|_{L^2(\Os)}
    \leq C \|g\|_{L^2(\Sigma)}\|\vmn\|_{H^{1/2}(\Os)}\\
    & \leq \frac{C}{(\nu mn)^{1/2}} \|g\|_{L^2(\Sigma)}^2 + \frac{1}{2}(\nu mn)^{1/2} \|\nabla \vmn\|_{L^2(\Os)}\|\vmn\|_{L^2(\Os)}\\
    & \leq \frac{C}{(\nu mn)^{1/2}} \|g\|_{L^2(\Sigma)}^2 
    + \frac{\nu m}{4}\|\nabla \vmn\|^2_{L^2(\Os)}  + \frac{n}{4}\|\vmn\|^2_{L^2(\Os)}.
\end{align*}
For the last term on the RHS of (\ref{errormix}) we have (\ref{nonlin}).

Combining the above inequalities and assuming $\|f\|_{H^{-1}(\Omega)}$ sufficiently small with respect to $\nu$, which allows to repeat (\ref{nonlin}) and absorb the last term of the RHS of (\ref{errormix}) by the LHS, we obtain 
\begin{equation*}
    \Vert \nabla \vmn\Vert^2_{L^2(\Of)}+\nu m\Vert \nabla \vmn 
    \Vert^2_{L^2(\Os)}+ n\Vert \vmn 
    \Vert^2_{L^2(\Os)}
    \leq \frac{C}{(\nu mn)^{1/2}} \Big( \|f\|_{L^2(\Os)}^2 + \|g\|^2_{L^2(\Sigma)} \Big),
\end{equation*}
from which we conclude (\ref{mix:3})-(\ref{mix:5}).
%The $m\cdot p(m)$ is equivalent to $m^2$, as $p(m)$ is linear function.
%Recall that $\vm= \um - u$ 
%and $u \equiv 0$ on $\Os$ we get
%\[\Vert \umn- u\Vert_{H^1(\Omega)}\leq m^{-1/2} C_1 (\Omega,g,f)\]
%and 
%\[\Vert \umn\Vert_{H^1(\Os)}\leq m^{-1}C_2 (\Omega,g,f)\]
\end{proof}

\section{Numerical simulations} \label{sec:num}
%\setstretch{1.5}
\renewcommand{\um}{\twoparam{u}{\viscosity}} 
\renewcommand{\un}{\twoparam{u}{\volume}} 

In this section, we present numerical experiments to investigate the dependence of the convergence rate of approximate solutions introduced in (\ref{NSvol})--(\ref{NSmix}) on the penalizing parameters $\viscosity$, $\volume$.
% to  verify if our bounds provided in Theorems \ref{thm:vol}, \ref{thm:visc}  and \ref{thm:mix} are sharp.  

To this end, we compute a two-dimensional flow around an obstacle in a fixed channel. In order to get broader insight, in addition to varying $\viscosity$ and $\volume$, we also change the shape and placement of obstacles. First, we consider a box-shaped obstacle touching the boundary in accordance with the theoretical results (see Section~\ref{sec:experim:box}). Next, in Section~\ref{sec:experim:taipei} we consider an obstacle with more complicated geometry to check how well penalization methods can handle cases with less regular solutions. Finally, in Section~\ref{sec:two obstacles} we relax the assumption that at least a part of each connected component of the interface $\Sigma$ must touch $\partial\Omega$ and consider a flow around two obstacles, where one of them is fully immersed in a fluid. 

The geometry and the input data closely follow \cite{jorge}, differing only in the number and shape of obstacles. Thus, all test cases are set up in a rectangular channel domain $\Omega = [0 ,L]\times [0 , H]\subset \R^2$, with $L=4$ and $H=2$. We assume the kinematic viscosity $\nu =1$, so that the Reynolds number $\Re = U H/\nu = 200$ with $U=100$.
Since we always assume no external forces, i.e. $f\equiv 0$ in (\ref{NS}), in our experiments we make a slight departure from the theoretical framework and consider, as in \cite{jorge}, a flow driven by nonhomogeneous boundary conditions. To be specific, on the left edge of $\Omega$, that is on $\Gamma_{in} = \{0\}\times [0, H]$, we prescribe an inflow Dirichlet boundary condition,
$$
u = (u_{in}, 0) \qquad \text{ on }\Gamma_{in}
$$
where $u_{in}$ is a parabolic profile
$$
u_{in} (x,y) = \dfrac{4U}{H^2} y\,(H - y).
$$
On the right edge, $\Gamma_{out} = \{L\}\times [0 , H]$, we prescribe the do-nothing boundary condition,
\[ \nu \partial_{n} u - p n = 0 \qquad \text{ on }\Gamma_{out},\]
where $n$ denotes the outer normal.
%\PK{The departure from the theoretical framework will be smaller if on the outflow part of the boundary an identical Dirichlet BC will be prescribed as on the inflow. The streamlines will look more or less the same as now, but maybe such setting would be easier to swallow for the reviewers?}
On all other parts of the boundary of the corresponding domain (i.e. $\Of$ in the case of ``real obstacle'' flow, or $\Omega$ otherwise), the no-slip boundary condition is imposed, $u=0$.

We have implemented our experimental framework in FEniCS package \cite{fenics}, discretizing Navier-Stokes equations with the $P_2-P_1$ Taylor-Hood finite element \cite{larson} on a triangular mesh in $\Omega$ (or in $\Of$ in the case of ``real obstacle'' flow), whose elements are aligned with the interface~$\Sigma$. The (unstructured) meshes have their diameter set to $h = 0.05$ and consist of roughly $8000$ triangular elements --- precise number depending on the selection of the obstacle --- and have been generated with the help of Gmsh software \cite{gmsh}.  
% and refined around the obstacles in order to be able to recover local spikes in the solution. 
The resulting nonlinear system of algebraic equations is then approximately solved by means of the Newton's method with accuracy $10^{-10}$.

In each experiment we solve numerically \eqref{NS} to compute the ``reference'' flow $(u,p)$ in $\Of$ around ``real obstacle'' $\Os$ and then compare it with the velocity components  $\um$, $\un$ and $\umn$ of approximate penalized numerical solutions to (\ref{NSvol}), (\ref{NSvisc}), (\ref{NSmix}), respectively. To get more insight, for each $\twoparam{u}{\app}$, where $\app\in\{\volume,\viscosity,\mixed\}$, and  for $10^1\leq m,n \leq 10^{10}$, we compute separately $L^2$ norm and $H^1$
seminorm of errors in the compound domain $\Omega$, 
$$
%\error{0}{\Omega}{\app} = 
\Vert u - \twoparam{u}{\app}\Vert_{L^2(\Omega)}, \quad
%\error{1}{\Omega}{\app} = 
\vert u - \twoparam{u}{\app}\vert_{H^1(\Omega)}
$$ 
and inside the obstacle domain $\Os$ as well, i.e.
$$
%\error{0}{\Os}{\app}  = 
\Vert \twoparam{u}{\app}\Vert_{L^2(\Os)}, \quad
%\error{1}{\Os}{\app} = 
\vert \twoparam{u}{\app}\vert_{H^1(\Os)}.
$$ 
(If possible, we also conduct some tests when either $n=0$ or $m=1$, i.e. for pure volume or viscosity penalization, respectively.)
Based on these measurements, we estimate empirical convergence rates as $m$ or $n$ increase towards the infinity in the above (semi)norms.

\subsection{A box-shaped obstacle touching the boundary}
\label{sec:experim:box}
\begin{figure}%[!tbp]
\centering
%\begin{minipage}[b]{1.0\textwidth}
\includegraphics[width=0.32\textwidth]{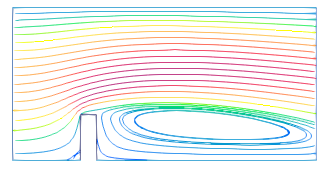}
\includegraphics[width=0.32\textwidth]{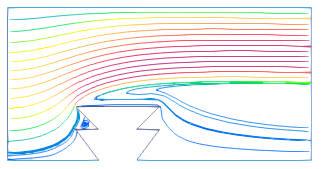}
\includegraphics[width=0.32\textwidth]{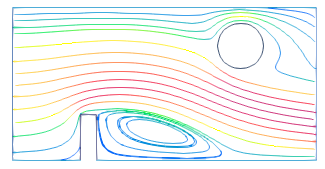}
\caption{From left to right: ,,real obstacle'' domains $\Of$ considered in Sections~\ref{sec:experim:box},~\ref{sec:experim:taipei} and~\ref{sec:two obstacles}, respectively, together with corresponding flow streamlines.} %Fluid domain $\Of$ marked in gray, obstacle $\Os$ marked in light blue.

\label{stream-real}
%\end{minipage}
\end{figure}
\begin{figure}%[!tbp]
\centering
%\begin{minipage}[b]{1.0\textwidth}
\includegraphics[width=0.32\textwidth]{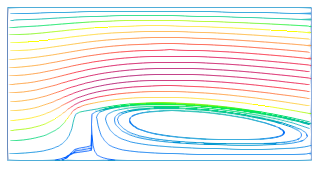}
\includegraphics[width=0.32\textwidth]{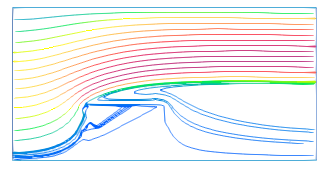}
\includegraphics[width=0.32\textwidth]{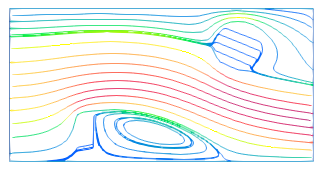}
\caption{Penalized solution streamlines in $\Omega$, approximating the flow depicted in \figref{stream-real}. The first two correspond to viscosity penalization with $m=10^5$; the last one comes from mixed penalization with $m=n=10^5$. Note how well the streamlines in the fluid domain $\Of$ approximate the actual flow around the obstacles, cf. \figref{stream-real}.}
\label{stream-penalized}
%\end{minipage}
\end{figure}

In this section, we experiment with an obstacle touching the boundary of the domain $ \Omega$, as required in~Theorem~\ref{thm:visc}.
%\todo{Do we have such an assumption in Theorem~\ref{thm:visc}?}
Precisely, we place a box-shaped constriction $\Os = [0.9,1.1]\times [0.0, 0.6]$ at the bottom of the channel, as shown in \figref{stream-real} (so the geometry corresponds to \cite{jorge} with the floating obstacle removed). 

\figref{h1_wall} presents the errors between the penalized and ``real obstacle'' solutions as a function of penalization parameter for all three types of approximation: volume, viscosity and a mixed one. For the latter, we prescribe $n=10^2\cdot m$ and so treat $m$ as the only independent parameter. 

The graphs indicate that penalized solutions do converge to the ``real obstacle'' solution at an asymptotically linear rate in all considered norms. In particular, the results suggest that the $H^1(\Os)$ estimate (\ref{visc:3}) is sharp. In all other combinations of norms and domains under consideration, experimental convergence rates are significantly better than predicted by Theorems~\ref{thm:vol}, \ref{thm:visc} and \ref{thm:mix} giving a hint that they may be suboptimal --- however, there is also some chance that the improved rates may be an artifact resulting from comparing numerical, not actual solutions. 
\begin{figure}%[!tbp]
\twoplot{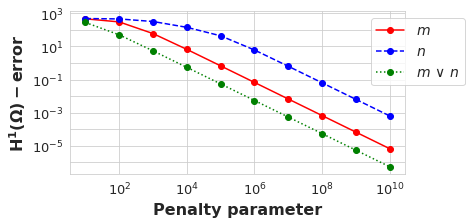}{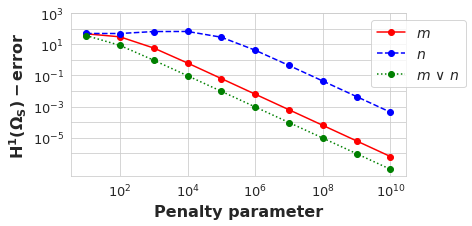}
\twoplot{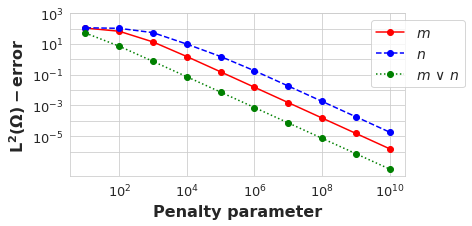}{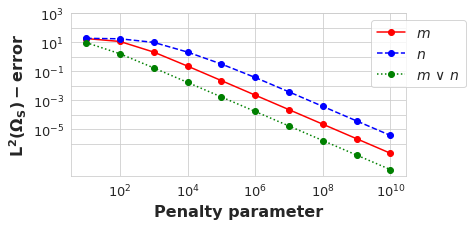}
%\PK{Captions below convergence plots such as in \figref{h1_wall} should preferably read "Penalty parameter $m$" or maybe just "$m$" instead of "Penalty parameters $m,n (10^2\cdot m)$" ---  "$n$" is described elsewhere.. Then x-axis tick labels should be adjusted.}
\caption{Approximation errors measured in $H^1$ seminorm (top) or $L^2$ norm (bottom) on $\Omega$ (left) and $\Os$ (right), for the experiment from Section~\ref{sec:experim:box}. The plots correspond to $\un$ (blue), $\um$ (red) and $\umn$ (green). For $\umn$, we set $n=10^2\cdot m$ and treat only $m$ as the independent penalization parameter (the $x$-axis corresponds to values of $m$ in this case).}
\label{h1_wall}
\twoplot{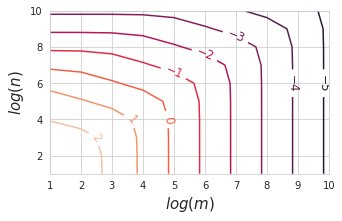}{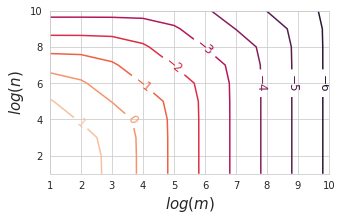}
\twoplot{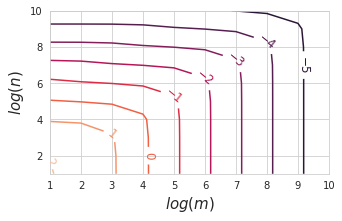}{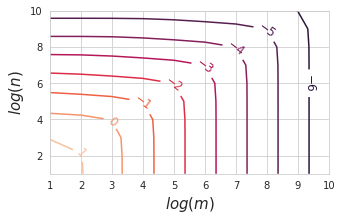}
\caption{Contour graphs of the logarithm of errors  in $H^1$ seminorm (top) or $L^2$ norm (bottom) measured on $\Omega$ (left) and $\Os$ (right), for the mixed penalization solution $\umn$  from Section~\ref{sec:experim:box}.}
\label{h1_lev_wall}
\end{figure}

Two other conclusions follow directly from convergence plots in \figref{h1_wall}. Firstly, for identical value of penalization parameter, viscosity approximation seems to result in an error smaller than the corresponding volume approximation (i.e. the former has a smaller multiplicative constant in front of $m^{-1}$). If \emph{both} penalization parameters are properly balanced, one can achieve an even smaller error when using mixed penalization. Secondly, the penalty parameter --- especially in the volume penalization case --- has to be large enough before the error starts to decrease at a linear speed; see e.g. the $H^1(\Os)$ error plot in \figref{h1_wall}. 

To get a better impression on how these two penalty parameters interact with each other in the mixed penalization case, in \figref{h1_lev_wall} we present logarithmic graphs of the errors for $10^1\leq m,n \leq 10^{10}$, with level lines corresponding to the magnitude of the error. Obviously, to get a prescribed level of accuracy, at least one penalty parameter must be sufficiently large. Moreover the observed error turns out to be roughly inversely proportional to $\max\{m, C n\}$ for some positive constant $C$. It also follows that for fixed $m$ (or $n$), the other penalization parameter must be large enough to start contributing to error improvement in a significant way.

\subsection{An obstacle with sharp corners}
\label{sec:experim:taipei}
For complex geometry fluid-structure problems, penalization methods in a fictitious domain may have a potential to be easier to implement and more efficient than other approaches. Therefore, in this section, we consider an obstacle featuring many acute or obtuse angles (see \figref{stream-real}), while leaving all other parameters unchanged, in order to check if such a complex shape --- which typically results in less regular solutions \cite{dauge} --- affects either the error level or the convergence rate. 

\figref{h1_taipei} shows the corresponding convergence histories. While we observe the same consistent pattern as in Section~\ref{sec:experim:box}, the approximation errors are roughly two orders of magnitude higher than those in \figref{h1_wall}, probably because the accurate solution is less regular. The level lines in \figref{h1_lev_taipei} are also qualitatively similar to in \figref{h1_lev_wall}, suggesting --- somewhat surprisingly --- that the order of approximation may be only weakly dependent, or even independent, on the smoothness of the solutions.

\begin{comment}
\begin{figure}%[!tbp]
\centering
\includegraphics[width=0.8\textwidth]{taipei.pdf}
\caption{ Domains for the experiment from Section~\ref{sec:experim:taipei}, depicted together with their triangulation. Fluid domain $\Of$ marked in green, obstacle $\Os$ marked in light orange. }
\label{mesh2}
\end{figure}
\end{comment}

\begin{figure}%[!tbp]
\twoplot{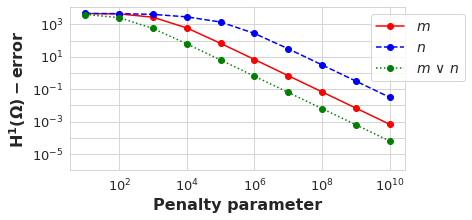}{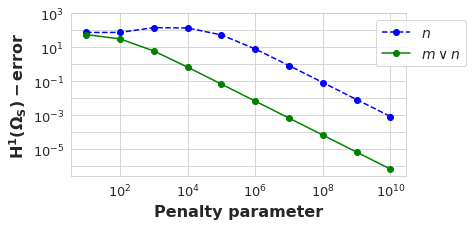}
\twoplot{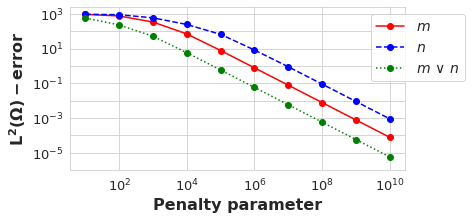}{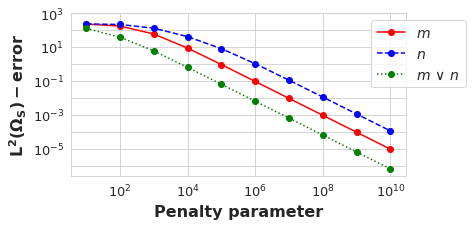}
\caption{Approximation errors measured in $H^1$ seminorm (top) or $L^2$ norm (bottom) on $\Omega$ (left) and $\Os$ (right), for the experiment from Section~\ref{sec:experim:taipei}. The plots correspond to $\un$ (blue), $\um$ (red) and $\umn$ (green). For $\umn$, we set $n=10^2\cdot m$ and put the values of $m$ on the $x$-axis, as before.}
\label{h1_taipei}
\twoplot{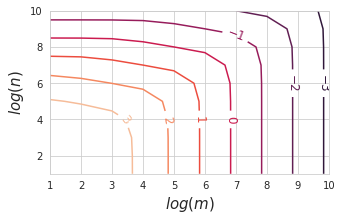}{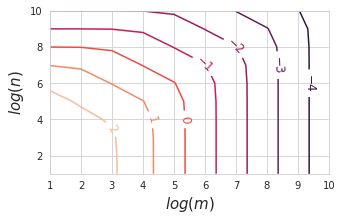}
\twoplot{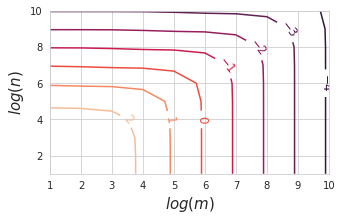}{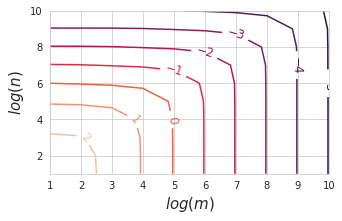}
\caption{Contour graphs of the logarithm of errors  in $H^1$ seminorm (top) or $L^2$ norm (bottom) measured on $\Omega$ (left) and $\Os$ (right), for the mixed penalization solution $\umn$  from Section~\ref{sec:experim:taipei}.}
\label{h1_lev_taipei}
\end{figure} 

\ifdraft
\begin{table}%[!h]
\centering
\begin{tabular}{ |p{0.7cm}|p{2.5cm}| p{2.5cm}|p{2.5cm}|}
\hline
 Penalty coef. &  $\Vert u - \um \Vert_{H^1(\Omega)} $ &  $\Vert u - \un \Vert_{H^1(\Omega)}$ &  $\Vert u - \umn \Vert_{H^1(\Omega)}  $ \\
\hline
$10$ & $5.510545\cdot 10^{2}$ & $ 5.793257\cdot 10^{2}$  & $5.415849\cdot 10^{2}$\\

$10^2$ & $3.815544\cdot 10^{2}$& $5.073848\cdot 10^{2}$& $3.625033\cdot 10^{2}$\\

$10^3$ & $9.830387\cdot 10^{1}$&$3.892458\cdot 10^{2}$&$9.064909\cdot 10^{1}$ \\ 

$10^4 $& $1.211198\cdot 10^{1} $&$2.386082\cdot 10^{2}$ &$1.102316\cdot 10^{1}$ \\

$10^5$& $1.242772\cdot 10^{0} $&$9.455582\cdot 10^{1}$&$1.128357\cdot 10^{0}$  \\

$10^6$ & $1.246052\cdot 10^{-1} $&$1.75999\cdot 10^{1}$&$1.131045\cdot 10^{-1}$ \\

$10^7$& $1.246381\cdot 10^{-2} $&$1.94404\cdot 10^{0}$&$1.131315\cdot 10^{-2}$\\

$10^8$& $1.246414\cdot 10^{-3} $&$1.964858\cdot 10^{-1}$&$1.131342\cdot 10^{-3}$\\

$10^9$& $1.246417\cdot 10^{-4} $&$1.966968\cdot 10^{-2}$&$1.131344\cdot 10^{-4}$\\

$10^{10}$& $1.246417\cdot 10^{-5}$&$1.967178\cdot 10^{-3}$&$1.131345\cdot 10^{-5}$ \\
\hline
\end{tabular}
\caption{Table corresponds to Fig \ref{Comp_8}}
\label{table:1}
\end{table}

\begin{table}%[!h]
\centering
\begin{tabular}{ |p{0.7cm}|p{2.5cm}| p{2.5cm}|p{2.5cm}|}
\hline
 Penalty coef. & $\Vert \um \Vert_{H^1(\Os)} $ &  $\Vert  \un \Vert_{H^1(\Os)}$ &  $\Vert \umn \Vert_{H^1(\Os)}  $ \\
\hline
$10$ & $8.249964\cdot 10^{1}$ & $ 9.585971\cdot 10^{1}$  & $8.019834\cdot 10^{1}$\\

$10^2$ & $4.237478\cdot 10^{1}$& $9.104529\cdot 10^{1}$& $3.876453\cdot 10^{1}$\\

$10^3$ & $8.749917\cdot 10^{0}$&$8.503317\cdot 10^{1}$&$7.813033\cdot 10^{0}$ \\ 

$10^4 $& $9.833153\cdot 10^{-1} $&$5.698188\cdot 10^{1}$ &$8.741463\cdot 10^{-1}$ \\

$10^5$& $9.95662\cdot 10^{-2} $&$3.675532\cdot 10^{1}$&$8.847293\cdot 10^{-2}$  \\

$10^6$ & $9.969154\cdot 10^{-3} $&$7.439083\cdot 10^{0}$&$8.858030\cdot 10^{-3}$ \\

$10^7$& $9.970409\cdot 10^{-4} $&$8.286719\cdot 10^{-1}$&$8.859105\cdot 10^{-4}$\\

$10^8$& $9.970534\cdot 10^{-5} $&$8.382179\cdot 10^{-2}$&$8.859213\cdot 10^{-5}$\\

$10^9$& $9.970547\cdot 10^{-6} $&$8.391845\cdot 10^{-3}$&$8.859223\cdot 10^{-6}$\\

$10^{10}$& $9.970549\cdot 10^{-7}$&$8.392813\cdot 10^{-4}$&$8.859225\cdot 10^{-7}$ \\
\hline
\end{tabular}
\caption{Table corresponds to Fig \ref{Comp_9}}
\label{table:2}
\end{table}
\fi

\subsection{Two obstacles, one surrounded by the fluid}
\label{sec:two obstacles}

\begin{comment}\begin{figure}%[!tbp]
\centering
\includegraphics[width=0.8\textwidth]{mesh_p.pdf}
\caption{ Domains for the experiment from Section~\ref{sec:experim:taipei}Section~\ref{sec:two obstacles}, depicted together with their triangulation. Fluid domain $\Of$ marked in yellow, obstacles $\Os^1$ and $\Os^2$  marked in light blue and light green. }
\label{mesh3}
\end{figure}
\end{comment}

Finally, we consider two obstacles, one of which is fully immersed in the fluid. The experiment replicates  (up to a horizontal symmetry) the setting from \cite{jorge}
$$
\Os^1 := [0.9,1.1 ]\times [0.0, 0.6] \quad \text{and} \quad \Os^2 := \left\lbrace (x,y)\in\R^2 : (x - 3.0 )^2 + (x - 1.5)^2= (0.3)^2 \right\rbrace, 
$$
so that $\Os = \Os^1 \cup \Os^2$, see \figref{stream-real}. Since $\Os^2$ does not touch the boundary of the domain $\Omega$, pure viscosity penalization will, in principle, not result in a sensible approximation. The reason is that in such case we only obtain $\|\nabla u\|\equiv 0$ in the limit, i.e. the solution is only constant (not necessarily vanishing) inside the immersed obstacle. Therefore in this case we restrict the comparison only to volume and mixed penalty approximations. 
\begin{figure}%[!tbp]
\twoplot{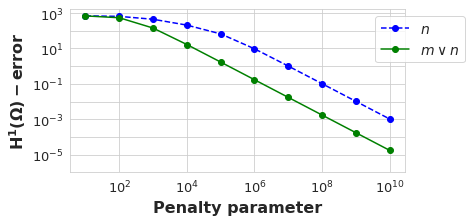}{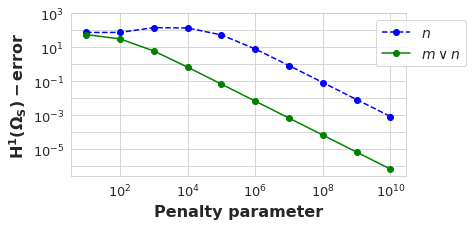}
\twoplot{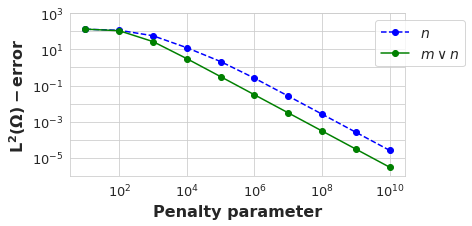}{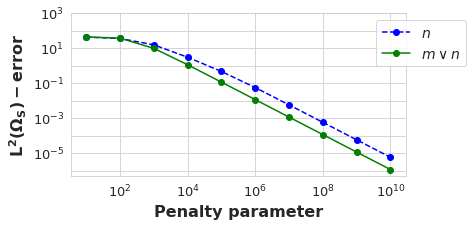}
\caption{Approximation errors measured in $H^1$ seminorm (top) or $L^2$ norm (bottom) on $\Omega$ (left) and $\Os$ (right), for the experiment from Section~\ref{sec:two obstacles}. The plots correspond to $\un$ (blue) and $\umn$ (green). For $\umn$, we set $n= m$.}
\label{h1_twoobs}
\twoplot{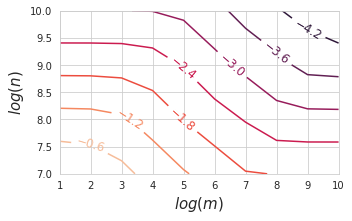}{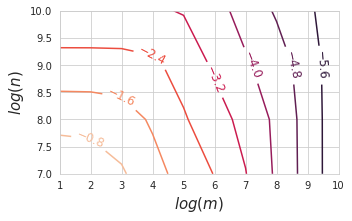}
\twoplot{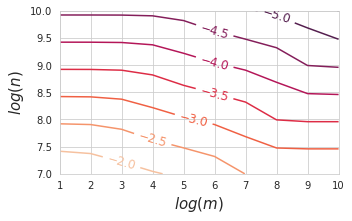}{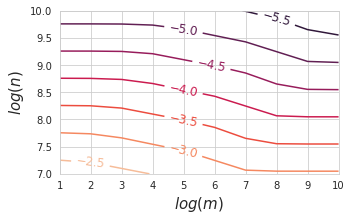}
\caption{Contour graphs of the logarithm of errors  in $H^1$ seminorm (top) or $L^2$ norm (bottom) measured on $\Omega$ (left) and $\Os$ (right), for the mixed penalization solution $\umn$  from Section~\ref{sec:two obstacles}.}
\label{h1_lev_twoobs}
\end{figure} 

Error graphs in \figref{h1_twoobs} confirm linear convergence in~$n$ and the possibility of further reduction of the error when using mixed penalization. \figref{h1_lev_twoobs}, which presents logarithmic graphs of the errors\footnote{We had to restrict ourselves to large volume penalty parameters, as our numerical solver struggled otherwise.} for $10^1 \leq m \leq 10^{10}$ and  $10^7 \leq n \leq 10^{10}$, indicates that to obtain a prescribed level of accuracy, the volume penalty parameter must be large enough; and the influence of the viscosity penalty is limited by the value of~$n$. This is in contrast to the case when all obstacles are touching $\partial\Omega$, when it is enough to increase only one of the penalties to improve the accuracy. 
 
\ifdraft
\begin{table}%[!h]
\centering
\begin{tabular}{ |p{0.7cm}|p{2.5cm}|p{2.5cm}|}
\hline
 Penalty parameters &  $\Vert u - \un \Vert_{H^1(\Omega)}$ &  $\Vert u - \umn \Vert_{H^1(\Omega)} $ \\
\hline
$10$  & $ 7.03678\cdot 10^{2}$  & $6.99473\cdot 10^{2}$\\

$10^2$ & $6.57444\cdot 10^{2}$& $5.50444\cdot 10^{2}$\\

$10^3$ &$4.50944\cdot 10^{2}$&$1.43554\cdot 10^{2}$ \\ 

$10^4 $& $2.10615\cdot 10^{2}$ &$1.67639\cdot 10^{1}$ \\

$10^5$&$6.63825\cdot 10^{1}$&$1.70392\cdot 10^{0}$  \\

$10^6$ &$9.40406\cdot 10^{0}$&$1.70671\cdot 10^{-1}$ \\

$10^7$&$9.85519\cdot 10^{-1}$&$1.70698\cdot 10^{-2}$\\

$10^8$& $9.90347\cdot 10^{-2}$&$1.70702\cdot 10^{-3}$\\

$10^9$&$9.90832\cdot 10^{-3}$&$1.70702\cdot 10^{-4}$\\

$10^{10}$ & $9.90881\cdot 10^{-4}$&$1.70702\cdot 10^{-5}$ \\
\hline
\end{tabular}
\caption{Table corresponds to Fig \ref{Comp_3}}
\label{table:3}
\end{table}

\begin{table}%[!h]
\centering
\begin{tabular}{ |p{0.7cm} |p{2.5cm}|p{2.5cm}|}
\hline
 Penalty parameters & $\Vert \un \Vert_{H^1(\Os)}$ &  $\Vert \umn \Vert_{H^1(\Os)}  $ \\
\hline
$10$ &  $ 7.29469\cdot 10^{1}$  & $5.48860\cdot 10^{1}$\\

$10^2$ & $7.32911\cdot 10^{1}$& $3.04753\cdot 10^{1}$\\

$10^3$&$1.40500\cdot 10^{2}$&$5.97586\cdot 10^{0}$ \\ 

$10^4$ & $1.33555\cdot 10^{2}$ &$6.63450\cdot 10^{-1}$ \\

$10^5$ & $5.15801\cdot 10^{1}$&$6.71108\cdot 10^{-2}$  \\

$10^6$ & $7.58036\cdot 10^{0}$&$6.71887\cdot 10^{-3}$ \\

$10^7$ & $7.97988\cdot 10^{-1}$&$6.71965\cdot 10^{-4}$\\

$10^8$ & $8.02268\cdot 10^{-2}$&$6.71973\cdot 10^{-5}$\\

$10^9$ & $8.02698\cdot 10^{-3}$&$6.71974\cdot 10^{-6}$\\

$10^{10}$ & $8.02742\cdot 10^{-4}$&$6.71974\cdot 10^{-7}$ \\
\hline
\end{tabular}
\caption{Table corresponds to Fig \ref{Comp_4}}
\label{table:4}
\end{table}
\fi

%%Contour

\ifdraft
\begin{table}%[!h]
\centering
\begin{tabular}{ |p{0.7cm}|p{2.5cm}|p{2.5cm}|}
\hline
 Penalty parameters &  $\Vert u - \un \Vert_{L^2(\Omega)}$ &  $\Vert u - \umn \Vert_{L^2(\Omega)} $ \\
\hline
$10$  & $ 1.29573\cdot 10^{2}$  & $ 1.30331\cdot 10^{2}$\\

$10^2$ & $ 1.12961\cdot 10^{2}$& $ 1.06910\cdot 10^{2}$\\

$10^3$ &$5.62956\cdot 10^{1}$&$ 2.67713\cdot 10^{1}$ \\ 

$10^4 $& $ 1.23009\cdot 10^{1}$ &$3.06350\cdot 10^{0}$ \\

$10^5$&$2.112170\cdot 10^{0}$&$ 3.10491\cdot 10^{-1}$  \\

$10^6$ &$2.55144\cdot 10^{-1}$&$ 3.10905\cdot 10^{-2}$ \\

$10^7$&$2.61801\cdot 10^{-2}$&$ 3.10946\cdot 10^{-3}$\\

$10^8$& $2.62504\cdot 10^{-3}$&$ 3.10950\cdot 10^{-4}$\\

$10^9$&$ 2.62574\cdot 10^{-4}$&$ 3.10951\cdot 10^{-5}$\\

$10^{10}$ & $2.62581\cdot 10^{-5}$&$3.10951\cdot 10^{-6}$ \\
\hline
\end{tabular}
\caption{Table corresponds to Fig \ref{Comp_7}}
\label{table:7}
\end{table}

\begin{table}%[!h]
\centering
\begin{tabular}{ |p{0.7cm} |p{2.5cm}|p{2.5cm}|}
\hline
 Penalty parameters & $\Vert \un \Vert_{L^2(\Os)}$ &  $\Vert \umn \Vert_{L^2(\Os)}  $ \\
\hline
$10$ &  $ 4.26973\cdot 10^{1}$  & $4.36019\cdot 10^{1}$\\

$10^2$ & $ 3.56756\cdot 10^{1}$& $3.71443\cdot 10^{1}$\\

$10^3$&$ 1.53869\cdot 10^{1}$&$ 9.73477\cdot 10^{0}$ \\ 

$10^4$ & $ 3.02264\cdot 10^{0}$ &$ 1.12297\cdot 10^{0}$ \\

$10^5$ & $ 4.63086\cdot 10^{-1}$&$ 1.13940\cdot 10^{-1}$  \\

$10^6$ & $ 5.53759\cdot 10^{-2}$&$ 1.14105\cdot 10^{-2}$ \\

$10^7$ & $5.68308\cdot 10^{-3}$&$ 1.14122\cdot 10^{-3}$\\

$10^8$ & $5.69856\cdot 10^{-4}$&$ 1.14123\cdot 10^{-4}$\\

$10^9$ & $5.70012\cdot 10^{-5}$&$ 1.14124\cdot 10^{-5}$\\

$10^{10}$ & $5.70028\cdot 10^{-6}$&$ 1.14124\cdot 10^{-6}$ \\
\hline
\end{tabular}
\caption{Table corresponds to Fig \ref{Comp_8}}
\label{table:8}
\end{table}
\fi

% \setstretch{1.0}
\section{Conclusions}

In this paper, we considered approximation methods of a fluid flow around an obstacle, using an approach based on penalization of the fluid motion inside the obstacle $\Os$ by means of either very large viscosity or friction (with corresponding penalty parameters $m$, $n$, respectively). 

Restricting ourselves to the case when the obstacle touches the boundary of the domain~$\Omega$ we proved that, under certain regularity assumptions, viscosity penalization converges at a linear rate inside the obstacle --- a result which our numerical experiments suggest is sharp. We also obtained bounds on the convergence speed in the mixed penalization case (i.e., when both volumetric and  viscosity terms are present) which, in the special case $m=n$ reduces to an optimal, linear rate on~$\Os$. Thanks to the aforementioned assumption, which stops the flow on a part of $\partial\Os$, this bound also compares favorably to \cite{angot2} and \cite{jorge}. 

On the other hand, our other convergence rate bounds --- in particular, those on entire $\Omega$ --- seem suboptimal: from computer simulations it follows that the error is roughly inversely proportional to $\max\{m, C n\}$ for some positive constant~$C$.
Interestingly, they also indicate that the convergence \emph{rate} is not influenced by the shape of the obstacle.
In addition, experiments show that for the same value of penalty parameter, viscosity penalization leads, in general, to smaller approximation errors (both in $\Os$ and $\Omega$) than the corresponding volume penalization. 

The overall picture significantly changes when the obstacle is fully immersed, so that it does not touch the boundary of~$\Omega$. The volume penalty parameter $n$ then becomes the leading force reducing the error, while the influence of the viscosity penalty $m$ is minor and limited by the value of~$n$.

\subsubsection*{Acknowledgements} The third (PBM) and fourth (TP) author have been partly supported by the Narodowe Centrum Nauki (NCN) grant No 2022/45/B/ST1/03432 (OPUS).

%\nocite{*}
%\bibliographystyle{plain}
%\bibliography{bibfile.bib}

\subsubsection*{Statements and Declarations}
The authors have no relevant financial or non-financial interests to disclose

\end{document}